\documentclass[11pt]{amsart}

\theoremstyle{plain}
\newtheorem{theorem}{Theorem}[section]
\newtheorem{lemma}[theorem]{Lemma}
\newtheorem{proposition}[theorem]{Proposition} 
\newtheorem{corollary}[theorem]{Corollary}
\newtheorem{conjecture}[theorem]{Conjecture}
\theoremstyle{remark}
\newtheorem{remark}[theorem]{Remark}

\theoremstyle{definition}
\newtheorem{definition}[theorem]{Definition} 
\newtheorem{notation}[theorem]{Notation} 

\newtheorem{example}[theorem]{Example}

\newcommand{\cD}{\mathcal{D}}

\newcommand{\cO}{\mathcal{O}}

\newcommand{\cE}{\mathcal{E}}

\newcommand{\C}{\mathbf{C}}
\newcommand{\N}{\mathbf{N}}
\newcommand{\Q}{\mathbf{Q}}
\newcommand{\Z}{\mathbf{Z}}
\newcommand{\Ps}{\mathbf{P}}
\newcommand{\bw}{\mathbf{w}}

\DeclareMathOperator{\sDEF}{DEF}

\DeclareMathOperator{\sing}{sing}
\DeclareMathOperator{\rk}{rank}

\DeclareMathOperator{\rank}{rank}

\DeclareMathOperator{\codim}{codim}
\DeclareMathOperator{\prim}{prim}
\DeclareMathOperator{\Pic}{Pic}
\DeclareMathOperator{\NL}{NL}
\DeclareMathOperator{\Gr}{Gr}

\numberwithin{equation}{section}

\title[Nodal varieties with defect]{Maximal families of nodal varieties with defect}

\author[R.~Kloosterman]{Remke Kloosterman}
\address{Institut f\"ur Mathematik, Humboldt-Universit\"at zu Berlin,
Unter den Linden 6, D-10099 Berlin, Germany} 
\email{klooster@math.hu-berlin.de}
\date{\today}
\thanks{The author thanks Ivan Cheltsov, Slawomir Cynk, Vincenzo Di Gennaro, Brendan Hassett and Orsola Tommasi for several comments on a previous version of this paper. The author would like to thank the referee for several suggestions to improve the presentation.
The  author is partially supported by DFG-grant KL 2244/2-1.}
\begin{document}

\begin{abstract}
In this paper we prove that a nodal hypersurface in $\Ps^4$ with defect has at least $(d-1)^2$ nodes, and if it has at most $2(d-2)(d-1)$ nodes and $d\geq 7$ then it contains either a plane or a quadric surface. Furthermore, we prove that a nodal double cover of $\Ps^3$ ramified along a surface of degree $2d$ with defect has at least $d(2d-1)$ nodes. We construct the largest dimensional family of nodal degree $d$ hypersurfaces in $\Ps^{2n+2}$ with defect for $d$ sufficiently large. 
%
\end{abstract}

\maketitle

\section{Introduction}
 Let $n\geq 3$ be an odd integer and $c$ be a positive integer. Let $1\leq w_0\leq \dots\leq w_{n+c}$ and $ 2\leq d_1\leq \dots\leq d_c$ 
be integers. For a nodal complete intersection  $X\subset \Ps(w_0,\dots,w_{n+c})=:\Ps$ of multidegree $d_1,\dots,d_c$ we define the defect of $X$ to be $h^{n+1}(X)-h^{n-1}(X)$.
 In this paper we consider the following problem of determining the minimal number of nodes to have positive defect. 
  In this generality the problem is too hard. 
In the sequel we  concentrate on the two  special cases: hypersurfaces in $\Ps^{n+1}$ and double solids, i.e., hypersurfaces of degree $2k$ in $\Ps(k,1,1,1,1)$.
In subsequent papers we will discuss the case of  three-dimensional complete intersections in $\Ps^{3+c}$
and of elliptic threefolds over $\Ps^2$, i.e., hypersurfaces of degree $6k$ in $\Ps(2k,3k,1,1,1)$.

We start by recalling some previous results on this problem.
For hypersurfaces in $\Ps^4$, Cheltsov showed that  the minimal number of nodes to have defect is $(d-1)^2$ \cite{ChelFac} and that if $X$ has defect and $(d-1)^2$ nodes then $X$ contains a plane \cite{ChelPlane}. This improves a previously known bound by Ciliberto and Di Gennaro \cite{CiroHS}. Ciliberto and Di Gennaro showed that if a hypersurface with defect has at most $2(d-2)(d-1)$ nodes and the defect is caused by a smooth surface, then $X$ contains either a plane or a quadric surface. 




To illustrate our methods we start by giving a new proof of Cheltsov's theorem
\begin{theorem}[{Cheltsov \cite{ChelFac, ChelPlane}}]
Let $X\subset \Ps^4$ be a nodal hypersurface of degree at least 3. Assume that $h^4(X)\geq 2$. Then $X$ has at least $(d-1)^2$ nodes. Moreover, if equality holds then $X$ contains a plane.
 \end{theorem}
We included this proof because it is a good illustration of our techniques and is significantly different from the one Cheltsov gave.

For fixed integers $n$ and $d$ let  $\sDEF_d\subset  \C[x_0,\dots,x_{2n+2}]_{d}$ be the locus of polynomials $f$ such that $V(f)$ is a nodal hypersurface with defect.
\begin{theorem} 
Fix $n\in \mathbf{N}, n\geq 2$. 
Then there exists a $D$ such that
\[ \codim L\geq    \binom{d+n+1}{n+1}-(n+1)(n+2)\]
holds for every $d>D$ and every irreducible component $L$ of $\sDEF_d$.

Moreover, if $n=1$ or Conjecture 1 of \cite{OtwHil} holds then we can take $D=2$ and any hypersurface in $\sDEF_d$ has at least $(d-1)^{n+1}$ nodes.
\end{theorem}

Otwinowska showed that \cite[Conjecture 1]{OtwHil} is implied by the Conjecture of Eisenbud, Green and Harris on the Hilbert functions of an ideal containing a complete intersection ideal.

We also consider the case of double covers of $\Ps^3$. In this case we recover a result by Cheltsov:
\begin{theorem}[{Cheltsov \cite{Chelwps}}] 
Let $f\in \C[x_0,x_1,x_2,x_3]$ be a homogeneous polynomial of degree $2d$ such that $V(f)$ is a nodal surface. Let $X:y^2=f$ be the double cover branched along $V(f)$. Suppose $h^4(X)>1$. Then $X$ has at least $d(2d-1)$ nodes.
 \end{theorem}

 In the case of hypersurfaces in $\Ps^4$ we prove the Ciliberto-Di Gennaro conjecture for $d\geq 7$:
 \begin{theorem}
  Let $X\subset \Ps^4$ be a nodal hypersurface of degree at least 7. Suppose that $X$ is non-factorial and that $X$ has at most $2(d-2)(d-1)$ nodes then
  either $X$ contains a plane and has $(d-1)^2$ nodes or $X$ contains a quadric surface and has at least $2(d-1)(d-2)$ nodes.
 \end{theorem}





We will now briefly discuss the strategy of proof. 
To reprove Cheltsov's result we use the following strategy: Let $I$ be the ideal of the nodes of $X$, where $X$ is a nodal hypersurface with defect. 
Let $H=V(\ell)$ be a general hyperplane then $X\cap H$ is smooth. In particular the ideal $I_H:=(I,\ell)$ defines an empty scheme. 
Since $X$ has defect this implies that the Hilbert polynomial of $I$ and the Hilbert function of $I$ are different in degree $2d-5$. From this it follows that $h_{I_H}(2d-4)\neq 0$ holds. 
Since the partials of the defining equation for $X$ are contained in $I_{d-1}$, it follows that $I_{d-1}$ has finitely many base points. Therefore $I_{d-1,H}$ is base point free. 
A combination of results from Macaulay and Gotzmann on which functions occur as  Hilbert functions of ideals yields that $h_{I_H}(k)\geq k+1$ for $k\leq d-2$ and that $h_{I_H}(k)\geq 2d-3-k$, for $ d-1\leq k \leq 2d-3$.  The observation $p_I\geq h_I(2d-4)\geq \sum_{k=0}^{2d-4} h_{I_H}(k)$ finishes the proof. The ideal $I_H$ is very similar to the ideal used by Green \cite{GreenF} to determine the largest component of the Noether-Lefschetz locus of surfaces in $\Ps^3$.

 The other proofs are variations of this idea. In the case of an  hypersurface in $\Ps^{2n+2}$ with $n\geq 2$, Macaulay's result is not strong enough to obtain the desired lower bound. In this case we use a result by Otwinowska \cite{OtwHil} instead. However this result bounds $h_{I_H}(k)$ only in a  certain interval. This is sufficient to detect the largest dimensional component, but not to establish the minimal number of nodes. If one assumes a conjecture from \cite{OtwHil}, then one obtains the desired lower bound for the number of nodes.
 The double solid case is very similar to the case of hypersurfaces in $\Ps^4$ and we will not comment on this. 
 
 
 The paper is organized as follows. In Section~\ref{sectMac} we recall several standard results on the Hilbert functions of ideals. In Section~\ref{secNodCI} we recall some standard results on the cohomology of nodal complete intersections. In particular, we present a formula to calculate the defect of a nodal hypersurface. In Section~\ref{sectHSDS} we prove the results for the hypersurfaces and in the double covers, except for the Ciliberto-Di Gennaro conjecture, which is proven in Section~\ref{secCil}. 
 
\section{Macaulay's and Green's result}\label{sectMac}
Let $S=\C[x_0,\dots,x_n]$ and let $I\subset S$ be a homogeneous ideal. Let $h_I$ be the Hilbert function of $I$, i.e., $h_I(k)=\dim (S/I)_k$.

Let $d\geq 1$ be an integer.
Let $c:=h_I(d)$. We can write $c$ uniquely as
\[c= \sum_{i=1}^d \binom{i+\epsilon_i}{i}\]
with $\epsilon_d\geq \epsilon_{d-1}\geq ... \geq \epsilon_1\geq -1$. We call this the \emph{(Macaulay) expansion} of $c$ in base $d$.
This expansion can be obtained inductively as follows: The number $\epsilon_d$ is the largest integer such that $\binom{d+\epsilon_d}{d}\leq c$. The numbers $\epsilon_i$ for $i<d$ are the coefficients in the expansion of $c-\binom{d+\epsilon_d}{d}$ in base $d-1$.

Using the Macaulay expansion of $c$ we define the following  numbers:
 \[c^{\langle d \rangle}:= \sum_{i=1}^d\binom{i+\epsilon_i+1}{i+1},\;  c_{\langle d \rangle}:=\sum_{i=1}^d \binom{i+\epsilon_i-1}{i} ,\; c_{*d}:=\sum_{i=2}^d \binom{i+\epsilon_i-1}{i-1}. \]
Note that  $c\mapsto c_{*d}$, $c\mapsto c^{\langle d \rangle}$ and  $c\mapsto c_{\langle d \rangle}$ are increasing functions in $c$.

Recall the following theorem by Macaulay:
\begin{theorem}[{Macaulay \cite{Mac}}]\label{thmMac} Let $V\subset S_d$ be a linear system and $c=\codim V$. Then the codimension of $V\otimes_{\C} S_1$ in $S_{d+1}$ is at most $c^{\langle d\rangle}$.
\end{theorem}

We apply this result mostly in the case where $V$ is the degree-$d$ part of an ideal $I$. In this case we can also obtain information on $h_I(d-1)$. 
\begin{corollary}\label{corMacDown}
Let $I\subset S$ be an ideal, $d\geq 2$ an integer and $c:=h_I(d)$. Then  
\[ h_I(d-1)\geq c_{* d} .\]
Moreover, if $\epsilon_{1}$ is nonnegative then $h_I(d-1)> c_{* d} $ holds.
\end{corollary}

For small $c$ we have the following Macaulay expansions in base $d$:
\begin{itemize}
\item For $c\leq d$ we have $\epsilon_d=\dots=\epsilon_{d-c+1}=0$ and $\epsilon_{d-c}=\dots=\epsilon_1=-1$. Hence $c^{\langle d\rangle}=c$.
\item For $d+1\leq c \leq 2d$ we have $\epsilon_d=1, \epsilon_{d-1}=\dots=\epsilon_{d-a}=0, \epsilon_{d-a-1}=\dots=\epsilon_1=-1$, where $a=c-d-1$.
Hence $c^{\langle d\rangle}=c+1$.
\item For $c=2d+1$ we have $\epsilon_d=\epsilon_{d-1}=1$ and all other $\epsilon_i$ equal $-1$. Hence $c^{\langle d \rangle}=2d+3=c+2$.
\end{itemize}

Applying the previous corollary repeatedly yields
\begin{corollary}\label{corMacLowDeg}
Let $I\subset S$ be an ideal, $d\geq 2$ an integer and $c:=h_I(d)$. For $0\leq k \leq d$ we have that
\[ h_{I}(k)\geq \left\{ \begin{array}{ll} \min(c,k+1) & \mbox{ if }c\leq d;\\
\min(k+(c-d),2k+1) &\mbox{ if } d+1\leq c\leq 2d;\\
2k+1 & \mbox{ if } c=2d+1.\end{array}\right.\]
\end{corollary} 

The following result will be used to detect the Hilbert polynomial of the ideal generated by $I_d$:
\begin{theorem}[{Gotzmann \cite{Gotz}}]\label{thmGotz}
Let $V\subset S_d$ be a linear system and let $J\subset S$ be the ideal generated by $V$. Set $c=h_J(d)$.  If $h_J(d+1)=c^{\langle d \rangle}$ then for all $k\geq d$ we have $h_J(k+1)=h_J(k)^{\langle k \rangle}$. In particular the Hilbert polynomial $p_J(t)$ of $J$ is given by
\[ \sum_{i=1}^d \binom{t+\epsilon_i}{t}\]
and the dimension of $V(J)$ equals $\epsilon_d$. 
\end{theorem}

We use this result mostly in the case where $c\leq d$:
\begin{corollary}\label{corGreen} Let $I\subset S$ be an ideal such that $h_I(d)\leq d$ and $I_{d+1}$ is base point free. Then for all $k\geq d$ we have $h_I(k+1)<h_I(k)$ or $h_I(k)=0$.
\end{corollary}
\begin{proof}
It suffices to prove the Corollary for $k=d$. If $h_I(d)=0$ then $h_I(k)=0$ for all $k\geq d$ and we are done. Suppose now that $h_I(d)>0$.
Let $I'$ be the ideal generated by $I_{d}$. From $h_{I'}(d)^{\langle d \rangle}=h_{I'}(d)$ it follows that $h_{I'}(d+1)\leq h_{I'}(d)$. If the inequality is strict then we are done, since $h_{I}(d+1)\leq h_{I'}(d+1)$.

Suppose now that $h_{I'}(d+1)=h_{I'}(d)$ holds.
Then Theorem~\ref{thmGotz} implies that the Hilbert polynomial of $I'$ equals $h_{I'}(d)$. Hence $I'_{d+1}$ has a base locus. Since $I'_{d+1}\subset I_{d+1}$ and $I_{d+1}$ is base point free it follows that \[h_I(d+1)<h_{I'}(d+1)\leq h_I(d).\]
\end{proof}

A final result of this type that we use is
\begin{theorem}[{Green, \cite{GreenHP}}]\label{thmGreenHyp}
Let $V\subset \C[x_0,\dots,x_n]_d$ be a linear system of codimension $c$. Let $H=\{\ell=0\}$ be a general hyperplane. Then the restriction of $V$ to $H$ has codimension at most $c_{\langle d\rangle}$ in $(\C[x_0,\dots,x_{n}]/\ell)_d$.
\end{theorem}

\section{Nodal complete intersections}\label{secNodCI}
\begin{notation}
Let $n=2k+1$ be a positive odd integer,  $c$ be a positive integer, and $(w_0,\dots,w_{n+c})$ a sequence of positive integers. Let us denote with $\Ps:= \Ps(w_0,\dots, w_{n+c})$ the associated weighted projective space. Let $S=\C[x_0,\dots,x_{n+c}]$ be the graded polynomial ring such that $\deg x_i=w_i$.
\end{notation}

\begin{definition}
We say that a codimension $c$ complete intersection $X\subset \Ps$ is \emph{a nodal complete intersection of codimension $c$}, if
\begin{enumerate}
\item for all $p\in \Ps_{\sing}\cap X$ we have that $X$ is quasi-smooth at $p$ and
\item for all $p\in X_{\sing} \setminus(\Ps_{\sing}\cap X)$ we have that $(X,p)$ is an $A_1$-singularity.
\end{enumerate}
Let $\Sigma$ denote the set $X_{\sing} \setminus(\Ps_{\sing}\cap X)$.
\end{definition}

\begin{proposition}\label{prpLef} Let $X\subset \Ps$ be a nodal complete intersection of codimension $c$ then for $i<n$ 
\[ \dim H^i(X)=\dim H^i(\Ps).\]
Moreover, for $i<n-1$ we have
\[ \dim H^i(X)=\dim H^{2n-i}(X).\]
\end{proposition}

\begin{proof}
The first equality follows from the Lefschetz hyperplane theorem \cite[Theorem 4.2.6]{Dim}. To prove the second equality we consider a partial resolution of singularities of $X$:

Since $X$ is quasismooth outside $\Sigma$ we have that for all $i\neq 2n$ and for all $p\in X\setminus \Sigma$ the group $H_p^i(X)$ vanishes.
Let $\tilde{X}$ be the blow up  of $X$ along $\Sigma$. Then $\tilde{X}$ is smooth along the exceptional divisor. In particular, for all $p\in \tilde{X}$ we have that $H^i_p(\tilde{X})=0$ if $i\neq 2n$. This implies that $\tilde{X}$ is $\Q$-homology manifold and satisfies Poincar\'e duality. 

Consider the Mayer-Vietoris sequence associated with the discriminant square \cite[Corollary-Definition 5.37]{PSbook}
\[ \dots \to H^i(X)\to H^{i}(\tilde{X}) \oplus H^i(\Sigma)\to H^i(E)\to H^{i+1}(X) \to \dots\]
This is an exact sequence of mixed Hodge structures.

The exceptional divisor $E$ is the disjoint union of $\#\Sigma$ smooth quadrics in $\Ps^n$. Thus its cohomology can be nonzero only in even degree between $0$ and $2n-2$. 
Let $E_j$ and $E_k$ be distinct irreducible components of $E$. For $i$ even between 2 and $2n-2$ consider $c_1(E_j)^{i/2}\in H^i(\tilde{X})$. Then $c_1(E_j)^{i/2}$  is mapped to zero in $H^i(E_k)$ and to a nonzero element of $H^i(E_j)$. If $i\neq n-1$ then $H^i(E_j)$ is one-dimensional and therefore the map $H^i(X)\to H^i(E)$ is surjective for even $i$, different from $0$ and $n-1$. From this it follows that the above long exact sequence splits in the following exact sequences:
\begin{itemize}
\item $ 0 \to H^0(X)\to H^0(\tilde{X}) \oplus H^0(\Sigma)\to H^0(E)\to 0$;
\item $0=H^i(X)\cong H^i(\tilde{X})$ for $i$ odd, different from $n$;
\item $ 0\to \Q=H^i(X)\to H^i(\tilde{X})\to H^i(E)\to 0$  for even $i$ different from  $0$ and $ n-1$;
\item$ 0 \to H^{n-1}(X)\to H^{n-1}(\tilde{X}) \to H^{n-1}(E) \to H^n(X)\to H^{n}(\tilde{X})$.
\end{itemize}

Since $E$ is a disjoint union of smooth quadrics it follows that   $h^i(E)=\#\Sigma=h^{2n-i}(E)$ for $i\neq 0, n-1,n+1, 2n$. From Poincar\'e duality it follows that $h^i(\tilde{X})=h^{2n-i}(\tilde{X})$ for all $i$. Combining this yields that that $h^i(X)=h^{2n-i}(X)$ for $i\neq 0, n-1,n+1,2n$.

To finish the proof, note that we showed that $h^0(X)=h^0(\tilde{X})=1$ and $h^{2n}(X)=h^{2n}(\tilde{X})=1$.
\end{proof}

The proof of the  above result suggests that   $h^{n+1}(X,\Q)$ may be strictly larger than $h^{n-1}(X,\Q)$.
\begin{definition}
The \emph{defect} $\delta$ of $X$ equals $h^{n+1}(X,\Q)-h^{n-1}(X,\Q)$.
\end{definition}

\begin{remark}
If $n=3$ then $\delta$ equals the rank of the group $\mathrm{CH}^1(X)/\Pic(X)$. Since this group is free, $\delta$ measures the failure of Weil divisors to be Cartier.
\end{remark}

\begin{lemma}\label{lemDefLoc} Let $X$ be a nodal complete intersection. Let $\cD$ be the equisingular deformation space of $X$. Then the locus 
\[ \{ X'\in \cD  \mid \delta(X')=\delta(X)\}\]
is a Zariski open subset of $\cD$.
\end{lemma}

\begin{proof} Let $(X_t)_{t\in U}$ be an equisingular deformation of $X$. Possibly after shrinking $U$,   we have that $X_t$ has the same number of nodes for all $t\in U$. Blowing up these nodes simultaneously yields a flat family  $\tilde{X_t}$ of smooth projective varieties. Hence $H^{n+1}(\tilde{X_t})$ is independent of $t$. Let $E$ be the exceptional divisor of the blow-up of a node, and let $s$ the number of nodes of $X$. As in the proof of Proposition~\ref{prpLef}  we can consider the Mayer-Vietoris sequence associated with the discriminant square. This time we take also into account the Hodge structures.  We obtain the following exact sequence
 \[ 0\to \Gr^W_{n+1} H^{n+1}(X_t)\to H^{n+1}(\tilde{X_t}) \to H^{n+1}(E)^{\oplus s} \to H^{n+2}(X_t). \]
Since $X_t$ is a nodal hypersurface and $n$ is odd we have that $H^{n+2}(X_t)=0$. Since all singularities of $X_t$ are nodes or induced by the ambient space it follows that  $H^{n+1}(X_t)$ has a Hodge structure of pure weight $n+1$. This yields
\[ h^{n+1}(X_t)=h^{n+1}(\tilde{X_t})-s\cdot  h^{n+1}(E).\]
Both terms on the right hand side are independent of $t$, hence so is $h^{n+1}(X_t)$. By Proposition~\ref{prpLef} we have $h^{n-1}(X_t)=1$ for all $t$ and hence  $\delta(X_t)=h^{n+1}(X_t)-h^{n-1}(X_t)=h^{n+1}(X_t)-1=h^{n+1}(X)-1=\delta(X)$.
\end{proof}

One can express $\delta$ in terms of the Hilbert function of the ideal of the nodes. 
Suppose now that $c=1$, i.e., $X$ is a hypersurface. 
Set $m:=\frac{n+1}{2}$.

The following result is \cite[Proposition 3.2]{DimBet}: 

\begin{proposition}\label{prpHS} Let $X\subset \Ps$ be a nodal hypersurface. Let $\Sigma\subset \Ps$ be the locus of the nodes of $X$. Then
\[ \delta(X)=\#\Sigma - \dim (S/I(\Sigma))_{md-\sum w_i}.\]
\end{proposition}

\section{Hypersurfaces with defect}\label{sectHSDS}
We will use the results from the previous section to reprove the following result by Cheltsov on the minimal number of nodes to have defect:
\begin{theorem}[{Cheltsov, \cite{ChelFac, ChelPlane}}]\label{thmHS} Let $X\subset \Ps^4$ be a nodal hypersurface of degree at least 3. Assume that $h^4(X)\geq 2$, i.e., that $X$ has defect. Then $X$ has at least $(d-1)^2$ nodes. If $X$ has precisely $(d-1)^2$ nodes then $X$ contains a plane.
 \end{theorem}
\begin{proof}
Without loss of generality we may assume that $X_H=X\cap V(x_4)$ is smooth. In particular,  none of the nodes of $X$ is contained in $V(x_4)$. Set $R=\C[x_0,x_1,x_2,x_3,x_4]$ and $S=\C[x_0,x_1,x_2,x_3]$. Let $I\subset R$ be the ideal of the nodes of $X$. Since $X$ has defect it follows from Proposition~\ref{prpHS} that $h_I(2d-5)<p_I(2d-5)$.

Let $I_H\subset S$ be the ideal obtained by substituting $x_4=0$ in $I$. 
From the fact that none of the nodes of $X$ is contained in $V(x_4)$ it follows that the following sequence is exact:
\[ 0 \to (R/I)_{k-1}\stackrel{x_4}{\to} (R/I)_k \to (S/I_H)_k \to 0.\]
If $h_{I_H}(2d-4)$ vanishes then we have $h_{I_H}(k)=0$ for $k\geq 2d-4$. In particular, $h_{I}(k)=h_I(k+1)$ for $k\geq 2d-5$. Since we know that $h_I(2d-5)<p_I(2d-5)$ this cannot be the case and hence $h_{I_H}(2d-4)>0$ holds.
 Fix now a codimension one subspace $W$ of $S_{2d-4}$ containing $(I_H)_{2d-4}$. Define $I'\subset S$ by  $I'_e=\{ g \mid g S_{2d-4-e} \subset W\}$ if $e\leq 2d-4$ and $I'_e=S_e$ for $e\geq 2d-3$.
Then $I'$ is an ideal, containing $I_H$. Moreover $S/I'$ is a Gorenstein ring with socle degree $2d-4$. In particular, $h_{I'}(k)=h_{I'}(2d-4-k)$.

Let $f$ be a defining polynomial for $X$. Since $(I_H)_{d-1}$ contains the partial derivative $\frac{\partial{f}}{\partial{x_i}}(x_0,x_1,x_2,x_3,0)$ for $i=0,\dots,3$ and $X\cap V(x_4)$ is smooth, we have that $I'_{d-1}$ is base point free. 

If $h_{I'}(k)<2d-3-k$ for some $k$ with $d-2\leq k \leq 2d-4$. Then from Corollary~\ref{corGreen} it would follow that $h_{I'}(2d-4)=0$, contradicting the fact that $h_{I'}(2d-4)=1$. Hence $h_{I'}(k)\geq 2d-3-k$ for every integer $k$ such that $d-2\leq k \leq 2d-4$.
Combining this information we obtain
\[ p_I=p_I(2d-4)\geq h_I(2d-4)=\sum_{i=0}^{2d-4} h_{I_H}(i)\geq\sum_{i=0}^{2d-4} h_{I'}(i) \geq (d-1)^2.\]

If $p_I=(d-1)^2$ then we have that $h_I$ equals the Hilbert function of a complete intersection of degree $(1,1,d-1,d-1)$. Since $I$ contains the partials of $f$ it follows that the linear system $|I_{d-1}|$  has finitely many base points. 

Now $I$ has two generators in degree 1 and two further generators in degree $d-1$. In particular these four generators defines a codimension four scheme and hence these four generators form a regular sequence. The Hilbert function of the ideal generated by these four forms equals the Hilbert function of $I$. Hence $I$ is a complete intersection ideal, generated by $f_1,f_2,f_3,f_4$, with $\deg(f_1)=\deg(f_2)=1, \deg(f_3)=\deg(f_4)=d-1$.

At each node of $X$ the polynomials $f_1,f_2,f_3,f_4$ induce a local system of coordinates. Since at each singular point of $X$ the polynomial $f$ vanishes up to order two it follows that $f$ is an element of the ideal generated by the $f_if_j$ with $i\leq j$. These forms have degree at most $d$ if and only if $i\leq 2$. In particular, $f$ is in the ideal generated by $f_1$ and $f_2$ and therefore contains the plane $f_1=f_2=0$.
\end{proof}

\begin{remark}\label{rmkLargestDim}
 The proof reveals also the following interesting observation. Suppose $I$ is the ideal of the nodes of a threefold of degree $d$ with defect. Then 
 \[ h_I(d)\geq \sum_{k=0}^d h_{I'}(k)\geq \frac{1}{2} (d^2+3d-10).\]
 Recall that $I_d$ is the tangent space to the equisingular deformation space of $X$ \cite{GrK}.  Hence it follows that any family of degree $d$ nodal hypersurfaces with defect has codimension at least $\frac{1}{2}(d^2+3d-10)$ in $S_d$. Moreover, if equality holds then the above proof shows that $X$ contains a plane.
 
 Consider now hypersurfaces containing a fixed plane $P$. They form a family of codimension $\frac{1}{2}(d+1)(d+2)$. Since the Grassmannian of planes in $\Ps^4$ has codimension 6 it follows that the total family has codimension $\frac{1}{2}(d^2+3d-10)$. A general element of this family is of the form $\ell_1f_1+\ell_2f_2$ with $\deg(\ell_i)=1$ and $\deg(f_i)=d-1$. In particular, a general element is a nodal hypersurface.
Hence the largest-dimensional family of nodal hypersurfaces with defect consists of hypersurfaces containing a plane.
\end{remark}

The bound we obtained for $h_{I_H}(k)$ (for $d-1\leq k\leq 2d-4$) is also used in some of the proofs for the explicit Noether-Lefschetz theorem for surfaces in $\Ps^3$ (e.g., see \cite{GreenF}). However, if $n>3$ then Corollary~\ref{corGreen} is insufficient to deduce the explicit Noether-Lefschetz theorem. Similarly, we were not able to deduce a good lower bound for the number of nodes to have defect from this Corollary.
To obtain an explicit Noether-Lefschetz theorem in higher (even) dimension Otwinowska  \cite{OtwHil} proved a result on the Hilbert function of ideals containing the ideal of a certain complete intersection.
This result seems still to be insufficient to obtain a sharp lower bound for the number of nodes to have defect. However, Otwinowska's result is strong enough to determine the largest component of the locus of nodal hypersurfaces with defect.
Moreover, if the famous conjecture \cite[Conjecture $V_m$]{EGH} of Eisenbud, Green and Harris on the Hilbert function of ideals containing a complete intersection holds true, then the result of Otwinowska is strong enough to deduce the minimal number of nodes.

\begin{notation}
Let us define $p_{n,d}=\binom{d+n+1}{n+1}-(n+1)(n+2)$. Then $p_{n,d}$  equals the Hilbert function of a complete intersection of multidegree $(1^{n+1},(d-1)^{n+1})$ evaluated in degree $d$, if $d>2$.
\end{notation}

 Consider a hypersurface $X \subset \Ps^{2n+2}$ of the form $\sum_{i=0}^n x_if_i$, with $\deg(f_i)=d-1$. If the $f_i$ are chosen sufficiently general then the singular locus is $x_0=\dots=x_n=f_0=\dots=f_n=0$. This is a complete intersection of multidegree $(1^{n+1},(d-1)^{n+1})$. The tangent space to the  equisingular deformation space  has codimension $p_{n,d}$ and an easy calculation shows that this space is nonreduced, i.e., the actual deformation space has the same codimension.
\begin{theorem} \label{thmHShigh}
Fix $n\in \N, n\geq 2$. Let $\sDEF_d\subset \C[x_0,\dots,x_{2n+2}]_d$ be the locus of nodal hypersurfaces with defect. Then there exists a $D$ such that if $d>D$ and  $L$ is an irreducible component of $\sDEF_d$ then $\codim L\geq p_{n,d}$.

Moreover, if Conjecture 1 of \cite{OtwHil} holds then we may take $D=2$ and any hypersurface in $\sDEF_d$ has at least $(d-1)^{n+1}$ nodes.
\end{theorem}
\begin{proof}
Let $X\in L$.
From Lemma~\ref{lemDefLoc} it follows that a general equisingular deformation of $X$ also has  defect, i.e., $L$ is also an irreducible component of the equisingular deformation space of $X$. 

Fix a general hyperplane $H$. Since $X$ has defect there is a class $\gamma$ in $H_{2n+2}(X,\Q)$ which is not the multiple of the intersection of classes of hyperplanes. The intersection product of $\gamma$ with $H$ yields a nonzero Hodge class in $H^{2n}(X_H,\Q)_{\prim}$. The Noether-Lefschetz locus of hypersurface of degree $d$ in $\Ps^{2n+1}$ parametrizes hypersurfaces having a nonzero Hodge class in $H^{2n}(X_H,\Q)_{\prim}$.
 In particular, we have a morphism from an open subset of $L$ to an irreducible component $\NL(\gamma_H)$ of this Noether-Lefschetz locus.
The differential of this map defines a map $d_H$ from the tangent space  $T_XL$ to the tangent space of $\NL(\gamma_H)$ at $X_H$.

Let $F$ be a defining polynomial for $X$. The tangent space $T_XL$ can be identified with the degree $d$ part of the saturation of the Jacobian ideal of $F$. Without loss of generality we may assume that $H=\{x_{2n+2}=0\}$.
We have that $J_d(F)|_{x_{2n+2}=0}$ is contained in $T_{X_H} \NL(\gamma_H)$. As explained in \cite{OtwHil} there exists an ideal $I\subset \C[x_0,\dots,x_{2n+1}]$, such that $ T_X\NL(\gamma_H)$ is contained in $I_d$ and $\C[x_0,\dots,x_{2n+1}]/I$ is an Artinian Gorenstein ring with socle degree $(n+1)d-2n-2$. Since $X_H$ is smooth we have that $I$ contains a complete intersection of multidegree $(d-1)^{2n+2}$.
Hence we can apply \cite[Th\'eor\`eme 1]{OtwHil}. From this it follows that there is a constant $D$ depending on $n$ such that for $d\geq D$  we have $\codim I_d\geq \binom{d+n}{n}-(n+1)^2$.

Let $J$ be the ideal of the nodes of $X$. Then 
\[ J(F(x_0,\dots,x_{2n+1},0)) \subset J|_{x_{2n+2}=0} \subset I.\]
If $\codim I_d=\binom{d+n}{n}-(n+1)^2$ holds then we have by \cite[Th\'eor\`eme 1]{OtwHil} that $I$ up to degree $d$ coincides with a complete intersection ideal of multidegree $(1^{n+1},(d-1)^{n+1})$. In this case $h_J(d)$ is at least
\[  \sum_{k=0}^d h_I(k)= \sum_{k=0}^d \binom{k+n}{n}-(n+1)-(n+1)^2=p_{n,d}.\]

If the codimension of $I_d$ is larger than $\binom{d+n}{n}-(n+1)^2$ then $\NL(\gamma_H)$ is different from  the component of $\NL$ parametrizing hypersurfaces containing an $n$-dimensional linear space. From \cite{OtwBig} it follows that for $d$ sufficiently large, the largest component of this type consists of hypersurfaces containing a quadric of dimension $n$.
This locus has codimension
\[ c_0:=\binom{d+n+1}{n+1}-\binom{d+n-1}{n+1}-\frac{3n^2+9n+4}{2}.\]

If  $n\geq 16$ then the Macaulay expansion of  $c_0$ equals
\[ \binom{d+n }{d}+\sum_{i=4}^{d-1} \binom{i+n-1}{i}+\binom{n-1}{3}+\binom{n-5}{2}+\binom{n-15}{1}.\]
For $n\leq 15$ we have that the the Macaulay expansion of $c_0$ equals
\[ \binom{d+n}{d}+\sum_{i=6}^{d-1} \binom{i+n-1}{i}+\sum_{i=1}^5 \binom{i+a_i}{i}\]
with $n-1 \geq a_7\geq a_6\geq\dots\geq a_1\geq -1$.

Suppose now that $n\geq 16$. Since $c\mapsto c_{<d>}$ increases with $c$ and 
 \[h_J(d)_{<d>}\geq h_{J_H}(d) \geq  c_0\] (Theorem~\ref{thmGreenHyp}) we have that $h_J(d)$ is at least
\[\binom{d+n+1}{d}+\sum_{i=4}^{d-1} \binom{i+n}{i}+\binom{3+n-3}{3}+\binom{2+n-6}{2}+\binom{1+n-15}{1}.\]
In particular, there exists a constant $C_n$ depending only on $n$ such that the right hand side equals $\binom{d+n+1}{d}+\binom{d+n}{d}-C_n$. Therefore we have that for $d$ sufficiently large $h_J(d)>p_{n,d}$ holds. If $n<16$ then a similar argument will yield the proof for large $d$.

Suppose now that \cite[Conjecture 1]{OtwHil} holds.  Let $I'\subset S_H$ be the ideal of a complete intersection of multidegree $(1^{n+1},(d-1)^{n+1})$. Then \cite[Conjecture 1]{OtwHil} implies $h_I(k)\geq h_{I'}(k)$ for all $k\leq (n+1)d-2n-1$. In particular,
\[ p_J\geq \sum_{k=0}^{nd-2n-2} p_{J_H}(k)\geq \sum_{k=0}^{nd-2n-2}  h_{I'}(k) = (d-1)^n.\]
and
\[h_J(d)\geq \sum_{k=0}^{d} p_{J_H}(k)\geq \sum_{k=0}^{d}  h_{I'}(k)=p_{n,d}.\]
\end{proof}

\begin{remark} Otwinowska shows in \cite{OtwHil} that \cite[Conjecture 1]{OtwHil} is implied by the Eisenbud-Green-Harris conjecture on the Hilbert function of ideals containing a complete intersection.
\end{remark}

We switch now to the case of double covers.

\begin{theorem} \label{thmDC} Let $f\in \C[x_0,x_1,x_2,x_3]$ be a squarefree polynomial of degree $2d$, such that $V(f)$ is a nodal surface. Let $X:y^2=f$ be the double cover branched along $f$. Suppose $h^4(X)>1$. Then $X$ has at least $d(2d-1)$ nodes. If $d
\geq 2$ holds and $X$ has precisely $d(2d-1)$ nodes then there  exist forms $\ell,g,h$ of degree $1$, $d$ and $2d-1$ respectively such that $f=\ell g+h^2$.
 \end{theorem}
\begin{proof}
Without loss of generality we may assume that $X_H=X\cap V(x_3)$ is smooth, in particular, none of the nodes of $X$ is contained in $V(x_3)$. Let $R=\C[x_0,x_1,x_2,x_3]$ and $S=\C[x_0,x_1,x_2]$. Let $I\subset R$ be the ideal of the nodes of $V(f)$. Note that the nodes of $X$ correspond one-to-one with the nodes of $V(f)$. Moreover, since $y$ is in the Jacobian ideal of $X$ we have that the Jacobian rings of $X$ and of $V(f)$ are isomorphic.

Since $h^4(X)\geq 2$ it follows from Proposition~\ref{prpHS} that we have $h_I(3d-4)<p_I(3d-4)$.
Let $I_H\subset S$ be the ideal obtained by substituting $x_3=0$ in $I$. 
Consider the exact sequence
\[ 0 \to (R/I)_{k-1}\stackrel{x_3}{\to} (R/I)_k \to (S/I_H)_k \to 0.\]
As in the proof of Theorem~\ref{thmHS} we obtain that $h_{I_H}(3d-3)>0$. Fix  a codimension one subspace $W$ of $S_{3d-3}$ containing $(I_H)_{3d-3}$. Define $I'\subset S$ by  $I'_e=\{ f \mid f S_{3d-3-e} \subset W\}$ if $e\leq 3d-3$ and $I'_e=S_e$ for $e\geq 3d-2$.
Then $I'$ is an ideal, containing $I_H$. Moreover $S/I'$ is a Gorenstein ideal with socle degree $3d-3$ and hence $h_{I'}(k)=h_{I'}(3d-3-k)$.

The linear system $I'_{2d-1}$ contains the partials of $f$ specialized at $x_3=0$ and since $X_H$ is smooth this linear system must be  base point free.  From Corollary~\ref{corGreen}   we obtain that $h_{I'}(k)\geq 3d-2-k$ for $2d-2\leq k \leq 3d-2$. 
Using Gorenstein duality it follows that $h_{I'}(k)\geq k+1$ for $k\leq d-1$.
Theorem~\ref{thmMac} implies that $h_{I'}(k)\geq d$ for $d\leq k\leq 2d-2$. 
Combining everything we obtain
\begin{eqnarray*} p_I&=&p_I(3d-3)\geq h_I(3d-3)=\sum_{i=0}^{3d-3} h_{I_H}(i)\\&\geq& 2\sum_{i=0}^{d-1}(i+1) +d(d-2)= d(d+1)+d(d-2)=d(2d-1).\end{eqnarray*} 

Suppose now that $p_I$ is exactly $d(2d-1)$. Then we have $h_I(1)=3$. In particular there is a linear form $\ell$ that vanishes at all the nodes. If $\ell$ is a factor of $f$ then we can write $f=\ell f_1$. All the nodes of $V(f)$ are contained in $V(\ell,f_1)$ which consists of $2d-1$ points. Since  we know that $X$ has at least $d(2d-1)$ nodes this cannot happen and therefore $\ell$ is not a factor of $f$. 

Assume  that $\ell=x_3$ and write $f=f_0(x_0,x_1,x_2)+x_3g(x_0,x_1,x_2,x_3)$. If $p$ is a node of $V(f)$ then $g$ vanishes at $p$ and $p$ is a double point of $f_0=0$. If $f_0$ contains a component with multiplicity at least three then $X$ contains a singularity which is not a node, in particular, we can write $f_0=f_1^2f_2$, such that $f_1$ and $f_2$ are coprime and both are squarefree. 
Hence the locus of the nodes of $V(f)$ consists of points $p$ such that $f_1(x_0,x_1,x_2)=g(x_0,x_1,x_2,0)=0$ together with points $p$ such that $g(x_0,x_1,x_2,0)=0$ and $p$ is a double point of $f_2$.

Denote with $e_i$ the degree of $f_i$. Then there are precisely $e_1(2d-1)$ points of the former type and at most $\frac{1}{2}(e_2-1)e_2$ points of the second type. Their sum is strictly less than $d(2d-1)$ if $e_1\neq 0, 2d$. If $e_1$ were $2d$ then the set of nodes of $V(f)$ is also the set of nodes of a (reducible) plane curve of degree $2d$. From \cite[Proposition 3.6]{EllSyz} it follows that then $h_I(k)=d(2d-1)$ holds for $k\geq 2d-2$, contradicting $h_I(3d-3)<d(2d-1)$. Hence $e_1=d$ and $f_0$ is  a square.
\end{proof}

\begin{example} In order to show that the bound $d(2d-1)$ for the number of nodes is sharp, consider $y^2=h^2+\ell g$ with $\deg(f)=d$ and $\deg(g)=2d-1$. Then for general $f,g,\ell$ the singular locus is a complete intersection of multidegree $(1,d,2d-1)$, i.e., it consists of $d(2d-1)$ points. Moreover $\ell=y-f=0$ defines a Weil divisor that is not $\Q$-Cartier and hence the double cover has defect.
\end{example}

\section{The Ciliberto-Di Gennaro conjecture}\label{secCil}
In this section we prove the following conjecture for $d\geq 7$:

\begin{conjecture} Let $X\subset \Ps^4$ be a non-factorial nodal threefold of degree $d$ with at most $2(d-2)(d-1)$ nodes then either $X$ contains a plane or a quadric surface and if $X$ contains a quadric surface then $X$ has precisely $2(d-2)(d-1)$ nodes.
\end{conjecture}
Note that for $d=1,2$ the conjecture is trivially true. If $d=3$ then $2(d-2)(d-1)=(d-1)^2$ and the statement follows from Theorem~\ref{thmHS}. Hence this conjecture remains open for $d=4,5,6$.

Note that in \cite{CiroHS}  Ciliberto-Di Gennaro proved a weaker form of this conjecture, namely they showed that if $X$ is non-factorial and has at most $2(d-2)(d-1)$ nodes then $X$ contains a plane, a quadric surface or a singular surface.

\begin{lemma}\label{lemdims} Let  $I\subset S:=\C[x_0,\dots,x_n]$ be a homogeneous ideal such that $S/I$ is Artinian Gorenstein of socle degree $N.$ Let $d_k$ be the smallest integer $t$ such that the dimension of the base locus of $I_t$ is at most $k$. Then 
\[ \sum_{k=-1}^{n-1} d_k \geq N+n+1\]
\end{lemma}
\begin{proof}
The ideal $I$ contains a complete intersection ideal $I'$ of multidegree $(d_{n-1},\dots,d_{-1})$. In particular, $I'_k=\C[x_0,\dots,x_n]_k$ for $k> \sum_{i=-1}^{n-1} (d_i-1)$. Since $I'_k\subset I_k$ and $I_N\neq \C[x_0,\dots,x_n]_N$ we have $\sum_{i=-1}^{n-1} (d_i-1)\geq N$.
\end{proof}

\begin{lemma}\label{lemHilba} Suppose $d\geq 6$. Let $X\subset \Ps^4$ be a nodal threefold of degree $d$ with at most $2(d-2)(d-1)$ nodes. Assume that $X$ has defect.
Let $J$ be the ideal of $X_{\sing}$, $H=\{\ell=0\}$ a general hyperplane, $J_H=(J,\ell)$. Let $I$ be an ideal containing $J_H$, such that $S/I$ is Artinian Gorenstein of socle degree $2d-4$.

If $h_I(d-4)\leq 2d-7$ then $X_{\sing}$ contains a subset which is a complete intersection of multidegree $(1,1,d-1,d-1)$ or of multidegree $(1,2,d-2,d-1)$.
\end{lemma}
\begin{proof} 
Let $S\subset X_{\sing}$ be a minimal subset such that the linear system of polynomials of degree $2d-5$ vanishing at $S$ has defect one. Let $J=I(S)$.
Then $h_{J_H}(2d-4)=1$ and $h_{J_H}(2d-5)=0$. Let $I$ be the ideal containing $J_H$, such that $S/I$ is Artinian Gorenstein of socle degree $2d-4$.

Suppose that $h_I(d)\leq 2d-4$. Then from the proof of \cite[Proposition 1.1]{VoiCon} in Section 1 of \textit{loc. cit.} it follows that there exists either a line $L$ or a conic $C$ such that $I_k=I(L)_k$ for $k\leq d-4$ or $I_k=I(C)_k$ for $k \leq d-4$.

Suppose first that $I_{d-4}$ is the degree $d-4$-part of the ideal of a line $L$.  
Using the notation of Lemma~\ref{lemdims} we have that $d_3=d_2=d_1=1$. Since the base locus of $I_{d-1}$ is empty it follows that $d_0\leq d_{-1}\leq d-1$. From Lemma~\ref{lemdims} it follows that $\sum_{i=-1}^{3} d_i\geq 2d+1$. In particular, $d_0=d_{-1}=d-1$ holds and $I$ contains a complete intersection ideal $I'$ of multidegree $(1,1,1,d-1,d-1)$. Since both $S/I$ and $S/I'$ are Artinian Gorenstein rings of socle degree $2d-4$ we have $I=I'$.

We are now going to show that the base locus $B$ of $J_{d-2}$ contains a plane. 
We claim that every component of $B$ has dimension at most 2, and that $B$ contains a component of dimension 2:

Suppose that first the base locus of $J_{d-2}$ would have dimension at least 3. Then $h_{J}(d-2)\geq h_{\Ps^3}(d-2)=\frac{1}{6}(d+1)d(d-1)$. This would imply that
\begin{eqnarray*} h_J(2d-4)&=&h_J(d-2)+\sum_{k=d-1}^{2d-4} h_{J_H}(k)\\&\geq& \frac{1}{6}(d+1)d(d-1)+\frac{1}{2}(d-1)(d-2)\\&>& 2(d-1)(d-2)\end{eqnarray*}
contradicting the fact that the length of $V(J)$ is at most $2(d-2)(d-1)$.

Hence the base locus $B$ of $J_{d-2}$ is of dimension at most two. The base locus of $I_{d-2}$ is a $L$ and is contained in $B\cap H$. Hence one of the irreducible components of $B$ is a plane $P$. 

We will now show that $P\subset X$. For this it suffices to show that $L\subset X_H$.
Recall that $J_d$ is the tangent space to space of deformations of $X$ where the points in $S$ deform to nodes. These equisingular deformations have then also defect. Hence $I_H\cap V(x_4)$ is contained in the tangent space $T \NL(\gamma)$ of the component $\NL(\gamma)$ 
of the Noether-Lefschetz locus of smooth surfaces of degree $d$ in $H$.
From standard arguments in Noether-Lefschetz theory it follows that $T \NL(\gamma)\otimes S_{d-4}$ has codimension at least one in $S_{2d-4}$ (see e.g., \cite{GreenF}). This space contains $I_d\otimes S_{d-4}$ and since $I$ is generated in degree $<d$ it follows $T \NL(\gamma)\otimes S_{d-4}=I_{2d-4}$. From the results in \cite{GreenF} it follows now that $T \NL(\gamma)=I_d$, and that the line $L$ is contained in $X_H$.

Suppose we are now in the case that $I_{d-4}$ is the degree $d-4$ part of the ideal of a conic. Without loss of generality we may assume that  the conic is defined by $x_0=x_1=f(x_2,x_3,x_4)=0$. Since $I_{d-1}$ is base point free  we can find  two further elements $f_1,f_2\in I$ of degree at most $d-1$, such that $x_0,x_1,f,f_1,f_2$  form a regular sequence of multidegree $(d_3,d_1,d_1,d_0,d_{-1})$. From Lemma~\ref{lemdims} it follows that $d_0+d_{-1}\geq 2d-3$. Since $d_{-1}\leq d-1$ we have two possibilities, namely $(d_0,d_{-1})$ equals $(d-2,d-1)$ or $(d-1,d-1)$.

Suppose first that $d_0=d-2$. Then  $I$ contains a complete intersection ideal  $I'$ of $(1,1,2,d-2,d-1)$. From $I_{2d-4}=I'_{2d-4}$ it follows that $I=I'$.
From this it follows that $J$ is contained in a complete intersection ideal of multidegree $(1,2,d-2,d-1)$. Using that $p_J\leq 2(d-2)(d-1)$ and that $p_{I'}=2(d-2)(d-1)$ it follows that $J=I'$ and that $X_{\sing}$ is a complete intersection of multidegree $(1,2,d-2,d-1)$.

Suppose now that $d_0=d-1$. Then $I$ contains a complete intersection ideal $I'$  of multidegree $(1,2,d-1,d-1)$. From e.g. \cite[Section 1]{OtwHil} it follows that if $I_1$ and $I_2$ are homogeneous ideals such that $I_1\subset I_2$ and both $S/I_1$ and $S/I_2$ are Artinian Gorenstein of socle degree $N+k$ and $N$ then $I_2=(I_1:F)$ for some form $F$ of degree $k$. We can apply this to $(I_1,I_2)=(I',I)$ and we find that
$I=(I':h)$ for some linear form $h$. Note that the base locus of $I'_{d-2}$ consists of a conic. If the base locus of $I_{d-2}$ is also a conic then we have that $h_{I}(k)=2k+1$ for $k\leq d-2$. Using Gorenstein duality we get that
\[ \sum_{k=0}^{2d-4} h_I(k)\geq 2\sum_{k=0}^{d-3}2k+1+2d-3=2(d-2)(d-1)+1\]
Since $J_H\subset I$ it follows that $p_J\geq h_J(2d-4)>2(d-2)(d-1)$, a contradicting. Hence the base locus of $I_{d-2}$ is not a conic. Since the base locus of $I_{d-4}$ is a conic and the base locus of $I_{d-2}$ is one dimensional we have that the conic is a union of two line lines $h_1h_2=0$ and one of the lines, say $h_1=0$ is contained in the base locus of $I_{d-2}$ and the linear form $h=h_2$. Recall that $I=(I':h_2)$. Since $h_1h_2\in I'$ it follows that $h_1\in I$ and therefore that the base locus of $I_1$ is contained in a line, a contradiction, hence $d_0=d-1$ is impossible. 
\end{proof}

\begin{lemma} \label{lemHilbb} Suppose $d\geq 7$. Let $X\subset \Ps^4$ be a nodal threefold of degree $d$ with defect.
Let $J$ be the Jacobian ideal of $X$, $H=\{\ell=0\}$ a general hyperplane, $J_H=(J,\ell)$. Let $I$ be an ideal containing $J_H$, such that $I$ is Artinian Gorenstein of socle degree $2d-4$.

If $h_I(d-4)> 2d-7$ then $X_{\sing}$ consists of at least $2 (d-1)(d-2)+1$ points.
\end{lemma}

\begin{proof} 
Let $h(k)=2k+1$ for $k\leq d-3$, $h(d-2)=2(d-2)$ and $h(k)=h(2d-4-k)$ for $d-1\leq k \leq 2d-4$. Then $h$ is the Hilbert function of a complete intersection ideal of multidegree $(1,2,d-2,d-1)$. 

Suppose that  for some $k$ we have that $h_I(k)>2k+1$. Then from Theorem~\ref{thmMac} it follows that $h_I(j)>2j+1=h(j)$ for $2\leq j \leq k$.

In our case we have that $h_I(d-4)>2d-7$. Hence  $h_I(k)>h(k)$ for $2\leq k \leq d-4$. Using Gorenstein duality we get $h_I(k)>h(k)$ for $d\leq k \leq 2d-6$. In particular, \[\sum_{k\neq d-3,d-2,d-1} h_I(k)-h(k)\geq 2(d-5)=2d-10.\]

From $h_I(d)\geq 2d-6$ it follows from  Theorem~\ref{thmMac} that $h_I(d-1) \geq 2d-7$. However, if $h_I(d-1)$ equals $2d-7$ then  $h_I(d)$ equals $2d-6$ and the Hilbert polynomial of the base locus of $I_{d-1}$ equals $k+(d-6)$ by Theorem~\ref{thmGotz}. In particular, the base locus of $I_{d}$ contains a line, which contradicts the fact that it is empty. Hence $h_I(d-1)\geq 2d-6$. Similarly, if $h_I(d-2)= 2d-7$ then $I_{d-1}$ has a base component. Since $I_{d-1}$ is base point free this is not possible.
This implies that $\sum_{k=d-3}^{d-1} h_I(k)-h(k)\geq -4$. In particular, if $2d-14>0$ then $\sum h_I(k)>\sum h(k)=2(d-2)(d-1)$. This finishes the proof in the case  $d>7$.

If $d=7$ then $\sum h_I(k)\geq 2(d-2)(d-1)$. Suppose now that equality holds. Then $h_I$ takes the following values $1,3,6,8,8,8,8,8,6,3,1$. This implies that the base locus in degree 1 and 2 is a plane. The base locus of $I_3$ has dimension at most one. From Lemma~\ref{lemdims} it follows that the dimension of the base locus is at least one. Since $h_I(3)=8$ it follows that the base locus of $I_3$ is the intersection of two plane cubics having either a line or a conic as a  common component.

From Lemma~\ref{lemdims} it follows that $I$ contains a complete intersection ideal $I'$ of multidegree $(1,3,4,6)$, $(1,3,5,5)$, $(1,3,5,6)$ or $(1,3,6,6)$.

In the first two case we would have that the socle degree of $S/I'$ is $10$. Since $S/I$ has also socle degree 10 this implies that  $I=I'$. However $I$ and $I'$ have different Hilbert functions, hence this is not the case.

If $I'$ is  of multidegree $(1,3,5,6)$ then there exists a linear form $h$ such that $I=(I':h)$. From this it follows that the base locus of $I'_k$ is contained in the base  locus of $I'_{k-1}$ union with $V(h)$. Note that the base locus of $I'_4$ is a cubic curve $C$ and the base locus of $I_3$ is the intersection of two cubics. Hence  the base locus of $I_3$ is a conic $Q=0$ together with a point and $V(h)$ is a component of $C$. Moreover we have that $Qh\in I'_3$. This implies that $Q\in (I':h)=I$, contradicting that $h_I(2)=6$, hence we can exclude this case.

If $I'$ is  of multidegree $(1,3,6,6)$ then there exists a quadratic form $h$ such that $I=(I':h)$. From this it follows that the base locus of $I'_k$ is contained in the base  locus of $I'_{k-2}$ union with $V(h)$. Note that the base locus of $I'_5$ is a cubic curve $C$ and the base locus of $I_3$ is the intersection of two cubics. Hence the base locus of $I_3$ is contains either a line $L=0$ or a conic $Q=0$, and this curve is a component of $C$. In the first case we have $Lh\in I'_3$ and by construction that $L\in I_1$, a contradiction. In the second case we have that $h=h_1h_2$ and $Qh_1\in I'_3$. This implies that $Qh_1h_2\in I'_4$ and that $Q\in (I':h)=I$. A contradiction.
\end{proof}

\begin{theorem}[Ciliberto-Di Gennaro conjecture] Suppose $d\geq 7$. Let $X$ be a nodal hypersurface of degree $d$, with at most $2(d-1)(d-2)$ nodes. Then one of the following holds
\begin{enumerate}
\item $X$ is factorial.
\item $X$ contains a plane and $X$ has at least $(d-1)^2$ nodes.
\item $X$ contains a quadric surface and $X$ has at least $2(d-1)(d-2)$ nodes.
\end{enumerate}
\end{theorem}

\begin{proof}
Suppose $X$ is not factorial.
It follows directly from Lemma~\ref{lemHilba} and~\ref{lemHilbb} that $X_{\sing}$ contains a complete intersection $\Sigma$ either of multidegree $(1,1,d-1,d-1)$ or of multidegree $(1,2,d-2,d-1)$. In the first case $\Sigma$ consists of $(d-1)^2$ points in the second case of $2(d-2)(d-1)$ points.

Let $f_1,f_2,f_3,f_4$ be the generators of $I(\Sigma)$, ordered by degree. Since the points of $\Sigma$ are on $X$ it follows that $f\in(f_1,f_2,f_3,f_4)$. Write $f=\sum h_if_i$. Since the points of $\Sigma$ are in $X_{\sing}$ and the $f_i$ form a system of local coordinates at each point of $\Sigma$ it follows that $h_i\in I(\Sigma)$. In particular $f$ is in the ideal generated by $f_if_j$. Such a product is of degree at most $d$ only if one of $i,j$ is at most 2. Hence $f\in (f_1,f_2)$ and therefore $X$ contains either a plane or a quadric surface, depending on the multidegree of the complete intersection.
\end{proof}

\begin{remark}
 This result also implies Theorem~\ref{thmHS}. However, in the above proof we used results from \cite{VoiCon}, which we can avoid in the proof of Theorem~\ref{thmHS}.
\end{remark}

\bibliographystyle{plain}
\bibliography{remke2}

\begin{thebibliography}{10}

\bibitem{Chelwps}
I.~Cheltsov.
\newblock Points in projective spaces and applications.
\newblock {\em J. Differential Geom.}, 81:575--599, 2009.

\bibitem{ChelFac}
I.~Cheltsov.
\newblock Factorial threefold hypersurfaces.
\newblock {\em J. Algebraic Geom.}, 19:781--791, 2010.

\bibitem{ChelPlane}
I.~A. Cheltsov.
\newblock On a conjecture of {C}iliberto.
\newblock {\em Sb. Math.}, 201:1069--1090, 2010.

\bibitem{CiroHS}
C.~Ciliberto and V.~Di~Gennaro.
\newblock Factoriality of certain threefolds complete intersection in
  {$\mathbb{ P}^5$} with ordinary double points.
\newblock {\em Comm. Algebra}, 32:2705--2710, 2004.

\bibitem{DimBet}
A.~Dimca.
\newblock Betti numbers of hypersurfaces and defects of linear systems.
\newblock {\em Duke Math. J.}, 60:285--298, 1990.

\bibitem{Dim}
A.~Dimca.
\newblock {\em Singularities and topology of hypersurfaces}.
\newblock Universitext. Springer-Verlag, New York, 1992.

\bibitem{EGH}
D.~Eisenbud, M.~Green, and J.~Harris.
\newblock Higher {C}astelnuovo theory.
\newblock {\em Ast\'erisque}, 218:187--202, 1993.
\newblock Journ{\'e}es de G{\'e}om{\'e}trie Alg{\'e}brique d'Orsay (Orsay,
  1992).

\bibitem{Gotz}
G.~Gotzmann.
\newblock Eine {B}edingung f\"ur die {F}lachheit und das {H}ilbertpolynom eines
  graduierten {R}inges.
\newblock {\em Math. Z.}, 158:61--70, 1978.

\bibitem{GreenHP}
M.~Green.
\newblock Restrictions of linear series to hyperplanes, and some results of
  {M}acaulay and {G}otzmann.
\newblock In {\em Algebraic curves and projective geometry ({T}rento, 1988)},
  volume 1389 of {\em Lecture Notes in Math.}, pages 76--86. Springer, Berlin,
  1989.

\bibitem{GreenF}
M.L. Green.
\newblock Components of maximal dimension in the {N}oether-{L}efschetz locus.
\newblock {\em J. Differential Geom.}, 29:295--302, 1989.

\bibitem{GrK}
G.-M. Greuel and U.~Karras.
\newblock Families of varieties with prescribed singularities.
\newblock {\em Comp. Math.}, 69:83--100, 1989.

\bibitem{EllSyz}
R.~Kloosterman.
\newblock Cuspidal plane curves, syzygies and a bound on the {MW}-rank.
\newblock {\em J. Algebra}, 375:216--234, 2013.

\bibitem{Mac}
F.~S. Macaulay.
\newblock Some {P}roperties of {E}numeration in the {T}heory of {M}odular
  {S}ystems.
\newblock {\em Proc. London Math. Soc.}, S2-26(1):531, 1927.

\bibitem{OtwHil}
A.~Otwinowska.
\newblock Sur la fonction de {H}ilbert des alg\`ebres gradu\'ees de dimension
  0.
\newblock {\em J. Reine Angew. Math.}, 545:97--119, 2002.

\bibitem{OtwBig}
A.~Otwinowska.
\newblock Composantes de petite codimension du lieu de {N}oether-{L}efschetz:
  un argument asymptotique en faveur de la conjecture de {H}odge pour les
  hypersurfaces.
\newblock {\em J. Algebraic Geom.}, 12:307--320, 2003.

\bibitem{PSbook}
C.~A.~M. Peters and J.~H.~M. Steenbrink.
\newblock {\em Mixed {H}odge structures}, volume~52 of {\em Ergebnisse der
  Mathematik und ihrer Grenzgebiete. 3. Folge.}
\newblock Springer-Verlag, Berlin, 2008.

\bibitem{VoiCon}
C.~Voisin.
\newblock Composantes de petite codimension du lieu de {N}oether-{L}efschetz.
\newblock {\em Comment. Math. Helv.}, 64(4):515--526, 1989.

\end{thebibliography}

\end{document}